\documentclass[12pt]{amsart}

\usepackage{amssymb, amscd, txfonts}
\usepackage[all]{xy}
\usepackage{graphicx}

\numberwithin{equation}{section}

\newtheorem{thm}{Theorem}[section]
\newtheorem{prop}[thm]{Proposition}
\newtheorem{lem}[thm]{Lemma}
\newtheorem{cor}[thm]{Corollary}

\theoremstyle{definition}
\newtheorem{defn}[thm]{Definition}

\theoremstyle{remark}
\newtheorem{rem}[thm]{Remark}

\renewcommand{\ker}{\operatorname{Ker}}

\newcommand{\Z}{\mathbb{Z}}
\newcommand{\Q}{\mathbb{Q}}

\newcommand{\K}{\mathbb{K}}

\DeclareMathOperator{\im}{Im}
\DeclareMathOperator{\Int}{Int}
\DeclareMathOperator{\out}{Out}
\DeclareMathOperator{\rank}{rank}

\DeclareMathOperator{\wh}{Wh}

\begin{document}

\title[Homology cylinders of higher-order]{Homology cylinders of higher-order}
\author[T.~Kitayama]{Takahiro KITAYAMA}
\address{Research Institute for Mathematical Sciences, Kyoto University, Kyoto, 606--8502 Japan}
\email{kitayama@kurims.kyoto-u.ac.jp}

\begin{abstract}
We study algebraic structures of certain submonoids of the monoid of homology cylinders over a surface and the homology cobordism groups, using Reidemeister torsion with non-commutative coefficients.
The submonoids consist of ones whose natural inclusion maps from the boundary surfaces induce isomorphisms on higher solvable quotients of the fundamental groups.
We show that for a surface whose first Betti number is positive, the homology cobordism groups are other enlargements of the mapping class group of the surface than that of ordinary homology cylinders.  
Furthermore we show that for a surface with boundary whose first Betti number is positive, the submonoids consisting of irreducible ones as $3$-manifolds trivially acting on the solvable quotients of the surface group are not finitely generated.
\end{abstract}

\maketitle

\section{Introduction}

Let $\Sigma_{g, n}$ be a compact oriented surface of genus $g$ with (possibly empty) $n$ boundary components.
We denote by $\mathcal{M}_{g, n}$ the mapping class group of $\Sigma_{g, n}$ which is defined to be the group of isotopy classes of orientation preserving homeomorphisms of $\Sigma_{g, n},$ where these isotopies are understood to fix $\partial \Sigma_{g, n}$ pointwise.

Homology cylinders were first introduced by Goussarov~\cite{G} and Habiro~\cite{Hab}, where these were referred to as \textit{homology cobordisms}, in their works on so-called \textit{clover} or \textit{clasper} surgery of $3$-manifolds developed for the study of finite-type invariants.
The set $\mathcal{C}_{g, n}$ of isomorphism classes of homology cylinders over $\Sigma_{g, n}$ naturally has a monoid structure by ``stacking''.
We denote by $\overline{\mathcal{C}}_{g, n}$ the submonoid consisting of isomorphism classes of irreducible ones as $3$-manifolds.
In \cite{GL, Le} Garoufalidis and Levine introduced the group $\mathcal{H}_{g, n}$ of smooth homology cobordism classes of homology cylinders over $\Sigma_{g, n},$ which can be seen as an enlargement of $\mathcal{M}_{g, n}.$
(See also \cite[Proposition 2.\ 4]{CFK}.)
These sets naturally act on $H_1(\Sigma_{g, n}; \Z),$ and we can consider substitutes $\mathcal{IC}_{g, n}, \mathcal{I} \overline{\mathcal{C}}_{g, n}, \mathcal{IH}_{g, n}$ of the Torelli subgroup $\mathcal{I}_{g, n}$ which are defined as the kernels of the actions.

It is a natural question which properties of $\mathcal{M}_{g, n}$ are carried over to $\mathcal{C}_{g, n}, \mathcal{H}_{g, n}.$
The following results contrast with the well-known facts that $\mathcal{M}_{g, n}$ is finitely presented, that $\mathcal{M}_{g, n}$ is perfect for $g \geq 3$~\cite{Po} and that $\mathcal{I}_{g, 0}$ and $\mathcal{I}_{g, 1}$ are finitely generated for $g \geq 3$~\cite{J}.
Morita~\cite{Mo2} showed by using his ``trace maps'' defined in \cite{Mo1} that the abelianization of $\mathcal{IH}_{g, 1}$ has infinite rank.
Goda and Sakasai~\cite{GS2} showed by using sutured Floer homology theory that $\overline{\mathcal{C}}_{g, 1}$ is not finitely generated if $g \geq 1$.
Cha, Friedl and Kim~\cite{CFK} showed by using abelian Reidemeister torsion that the abelianization of $\mathcal{H}_{g, n}$ contains a direct summand isomorphic to $(\Z / 2)^{\infty}$ if $(g, n) \neq (0, 0), (0, 1)$ and one isomorphic to $\Z^{\infty}$ if $n > 1$, and that the abelianization of $\mathcal{IH}_{g, n}$ contains a direct summand isomorphic to $(\Z / 2)^{\infty}$ if $(g, n) \neq (0, 0), (0, 1)$ and one isomorphic to $\Z^{\infty}$ if $g > 1$ or $n > 1.$

We set $\Gamma_m := \pi_1 \Sigma_{g, n} / (\pi_1 \Sigma_{g, n})^{(m+1)}$ for each $m \geq 0,$ where $(\pi_1 \Sigma_{g, n})^{(m)}$ is the derived series of $\pi_1 \Sigma_{g, n}.$
The derived series $G^{(m)}$ of a group $G$ is defined inductively by $G^{(0)} := G$ and $G^{(m+1)} := [G^{(m)}, G^{(m)}].$
In this paper for given $m,$ we introduce \textit{homology cylinders of order $m$} over $\Sigma_{g, n}$, which are characterized as homology cylinders over $\Sigma_{g, n}$ satisfying that the marking embeddings from $\Sigma_{g, n}$ to the boundary of the underlying manifold $M$ induce isomorphisms $\Gamma_m \to \pi_1 M / (\pi_1 M)^{(m+1)}.$
We denote by $\mathcal{C}_{g, n}^{(m)}$ and $\overline{\mathcal{C}}_{g, n}^{(m)}$ the submonoids of  $\mathcal{C}_{g, n}$ and $\overline{\mathcal{C}}_{g, n}$ consisting of isomorphism classes of homology cylinders of order $m.$
These naturally give filtrations of $\mathcal{C}_{g, n}$ and $\overline{\mathcal{C}}_{g, n}.$
We also define an appropriate smooth homology cobordism group $\mathcal{H}_{g, n}^{(m)}$ in this context, which can be also seen as an enlargement of $\mathcal{M}_{g, n}.$
There are a natural homomorphism $\mathcal{C}_{g, n}^{(m)} \to \out(\Gamma_m)$ and the induced homomorphisms $\overline{\mathcal{C}}_{g, n}^{(m)} \to \out(\Gamma_m), \mathcal{H}_{g, n}^{(m)} \to \out(\Gamma_m).$
We use the notation $\mathcal{IC}_{g, n}^{(m)}, \mathcal{I} \overline{\mathcal{C}}_{g, n}^{(m)}, \mathcal{IH}_{g, n}^{(m)}$ for the kernels, which are substitutes of $\ker(\mathcal{M}_{g, n} \to \out(\Gamma_m)).$
The filtration
\[ \dots \subset \mathcal{IC}_{g, n}^{(m+1)} \subset \mathcal{IC}_{g, n}^{(m)} \subset \dots \subset \mathcal{IC}_{g, n}^{(1)} \subset \mathcal{IC}_{g, n} \]
and the sequence of the homomorphisms
\[ \dots \to \mathcal{IH}_{g, n}^{(m+1)} \to \mathcal{IH}_{g, n}^{(m)} \to \dots \to \mathcal{IH}_{g, n}^{(1)} \to \mathcal{IH}_{g, n} \]
can be seen as alternatives for the derived series of the Johnson filtrations of $\mathcal{C}_{g, n}, \mathcal{H}_{g, n}$~\cite{Hab, GL} for the lower central series.

Our purpose is to investigate algebraic structures of these objects by using Reidemeister torsion over skew fields as an analogue of the work of Cha, Friedl and Kim~\cite{CFK}.
Such torsion invariants are known by Friedl in \cite{Fri} to be essentially equal to higher-order Alexander polynomials introduced for knots by Cochran~\cite{C} and extended to $3$-manifolds by Harvey~\cite{Har} and Turaev~\cite{T3}.
We first construct the Reidemeister torsion homomorphisms
\begin{align*}
\mathcal{C}_{g, n}^{(m)} &\to (\Q(\Gamma_m)_{ab}^{\times} / \pm \Gamma_m) \rtimes \out(\Gamma_m), \\
\mathcal{H}_{g, n}^{(m)} &\to (\Q(\Gamma_m)_{ab}^{\times} / \pm \Gamma_m \cdot \langle q \bar{q} \rangle) \rtimes \out(\Gamma_m),
\end{align*}
where $\Q(\Gamma_m)$ is the classical right ring of quotients $\Q[\Gamma_m] (\Q[\Gamma_m] \setminus 0)^{-1}$ of $\Q[\Gamma_m]$ and $\bar{\cdot} \colon \Q(\Gamma_m)_{ab}^{\times} \to \Q(\Gamma_m)_{ab}^{\times}$ is the induced involution by $\gamma \mapsto \gamma^{-1}$ for $\gamma \in \Gamma_m.$ 
(See Corollaries \ref{cor_H1}, \ref{cor_H2}, \ref{cor_H3}, \ref{cor_H4}.)
Moreover, we prove the following theorems, which establish own interests of the objects.
(See also Lemmas \ref{lem_TC}, \ref{lem_TH}.)
\begin{thm}[Theorem \ref{thm_C}] \label{thm_I2}
(i) $\mathcal{I} \overline{\mathcal{C}}_{0, 2}^{(1)} \neq \mathcal{I} \overline{\mathcal{C}}_{0, 2}.$

(ii) $\mathcal{I} \overline{\mathcal{C}}_{1, 0}^{(1)} \neq \mathcal{I} \overline{\mathcal{C}}_{1, 0}.$

(iii) If $(g, n) \neq (0, 0), (0, 1), (0, 2), (1, 0),$ then $\mathcal{I} \overline{\mathcal{C}}_{g, n}^{(m+1)} \neq \mathcal{I} \overline{\mathcal{C}}_{g, n}^{(m)}$ for all $m.$
\end{thm}

\begin{thm}[Theorem \ref{thm_H}] \label{thm_I3}
If $(g, n) \neq (0, 0), (0, 1),$ then the homomorphisms $\mathcal{H}_{g, n}^{(m)} \to \mathcal{H}_{g, n}, \mathcal{IH}_{g, n}^{(m)} \to \mathcal{IH}_{g, n}$ are not surjective for $m > 0.$
\end{thm}

Finally, we prove the following theorem, and give an observation on an approach for whether $\mathcal{IH}_{g, n}^{(m)}$ is in general finitely generated or not.
\begin{thm}[Corollary \ref{cor_A}] \label{thm_I4}
If $n > 0$ and $(g, n) \neq (0, 1), (0, 2),$ then $\mathcal{I} \overline{\mathcal{C}}_{g, n}^{(m)}$ is not finitely generated for all $m.$
\end{thm}
In fact we show that the group completion of $\mathcal{I} \overline{\mathcal{C}}_{g, n}^{(m)}$ has an abelian group quotient of infinite rank.
These can be regarded as an analogue of the question whether $\ker(\mathcal{M}_{g, n} \to \out(\Gamma_m))$ is finitely generated or not.
It is worth pointing out that the technique in Section \ref{sec_5} to detect non-triviality of elements in $\Q(\Gamma_m)_{ab}^{\times} / \pm \Gamma_m$ has multiplicity of use and could be useful  also in other applications of non-commutative Reidemeister torsion. 

This paper is organized as follows.
In Section \ref{sec_2} we define homology cylinders of order $m$ and smooth homology cobordisms of them.
Section \ref{sec_3} establishes the Reidemeister torsion homomorphisms of $\mathcal{C}_{g, n}^{(m)}, \mathcal{H}_{g, n}^{(m)}$ and contains a proof of Theorem \ref{thm_I3}. 
Section \ref{sec_4} provides a way to construct homology cylinders of order $m$ from knots in $S^3$ by performing surgery and computations of Reidemeister torsion of them.
Here we prove Theorem \ref{thm_I2}.
Finally, we prove Theorem \ref{thm_I4} and discuss an approach for $\mathcal{IH}_{g, n}^{(m)}$ in Section \ref{sec_5}.

In this paper all homology groups and cohomology groups are with respect to integral coefficients unless specifically noted. \\

\section{Definitions} \label{sec_2}
\subsection{The monoids of homology cylinders of higher-order}
We begin with introducing the monoids of homology cylinders of higher-order, which give a filtration of the monoid of ordinary homology cylinders introduced in \cite{G}, \cite{Hab}.
See \cite{HM}, \cite{S3} for more details on homology cylinders. 

To simplify notation we often write $\Sigma, \pi$ instead of $\Sigma_{g, n}, \pi_1 \Sigma_{g, n},$ respectively.

\begin{defn}
For an integer $m \geq 0,$ a \textit{homology cylinder} $(M, i_{\pm})$ of order $m$ over $\Sigma$ is defined to be a compact oriented $3$-manifold $M$ together with embeddings $i_+, i_- \colon \Sigma \to \partial M$ satisfying the following:

(i) $i_+$ is orientation preserving and $i_-$ is orientation reversing,

(ii) $\partial M = i_+(\Sigma) \cup i_-(\Sigma)$ and $i_+(\Sigma) \cap i_-(\Sigma) = i_+(\partial M) = i_-(\partial M),$

(iii) $i_+|_{\partial \Sigma} = i_-|_{\partial \Sigma},$

(iv) $(i_+)_*, (i_-)_* \colon \Gamma_m \to \pi_1 M / (\pi_1 M)^{(m+1)}$ are isomorphisms. \\
Two homology cylinders $(M, i_{\pm}), (N, j_{\pm})$ are called isomorphic if there exists an orientation preserving homeomorphism $f \colon M \to N$ satisfying $j_{\pm} = f \circ i_{\pm}.$
We denote by $\mathcal{C}_{g,n}^{(m)}$ the set of all isomorphism classes of homology cylinders of order $m$ over $\Sigma_{g, n}.$
\end{defn}

\begin{rem} \label{rem_H}
By the Hurewicz theorem and a standard argument on homology groups it follows from the condition (iv) that $(i_+)_*, (i_-)_* \colon H_*(\Sigma) \to H_*(M)$ are isomorphism.
In particular, a homology cylinder of order $0$ is nothing but an ordinary homology cylinder, i.e., $\mathcal{C}_{g,n}^{(0)} = \mathcal{C}_{g,n}.$
\end{rem}

For $\psi \in \mathcal{M}_{g, n},$ we define $M(\psi)$ to be the homology cylinder $\Sigma \times [0, 1] / \sim$ (of order $m$ for all $m$) equipped with $(i_+ = id \times 1, i_- = \psi \times 0),$ where $(x, s) \sim (x, t)$ for $x \in \partial \Sigma$ and $s, t \in [0, 1].$
A product operation on $\mathcal{C}_{g, n}^{(m)}$ is given by stacking:
\[ (M, i_{\pm}) \cdot (N, j_{\pm}) := (M \cup_{i_- \circ (j_+)^{-1}} N, i_+, j_-), \]
which turns $\mathcal{C}_{g, n}^{(m)}$ into a monoid from the following lemma. 
The unit is given by $M(id).$
For all $m,$ $\mathcal{C}_{g, n}^{(m+1)}$ is a submonoid of $\mathcal{C}_{g, n}^{(m)}.$
The correspondence $\psi \mapsto M(\psi)$ gives a monoid homomorphism $\mathcal{M}_{g, n} \to \mathcal{C}_{g, n}^{(m)}$ for each $m.$

\begin{lem} \label{lem_P}
For $(M, i_{\pm}), (N, j_{\pm}) \in \mathcal{C}_{g, n}^{(m)},$ $(M, i_{\pm}) \cdot (N, j_{\pm}) \in \mathcal{C}_{g, n}^{(m)}.$
\end{lem}

\begin{proof}
We only need to check that $(M, i_{\pm}) \cdot (N, j_{\pm})$ satisfies the condition (iv).

By the van Kampen theorem $\pi_1 (M \cup_{i_- \circ (j_+)^{-1}} N) = \pi_1 M *_{\pi} \pi_1 N.$
Let $\Pi$ be the subgroup normally generated by elements of $(\pi_1 M)^{(m+1)} *_{\pi^{(m+1)}} (\pi_1 N)^{(m+1)}.$
Then
\[ (\pi_1 M *_{\pi} \pi_1 N) / \Pi \cong (\pi_1 M / (\pi_1 M)^{(m+1)}) *_{\Gamma_m} (\pi_1 N / (\pi_1 N)^{(m+1)}) \cong \Gamma_m. \]
Since $\Pi$ is a subgroup of $(\pi_1 M *_{\pi} \pi_1 N)^{(m+1)},$ $(i_+)_*, (j_-)_* \colon \Gamma_m \to (\pi_1 M *_{\pi} \pi_1 N) / (\pi_1 M *_{\pi} \pi_1 N)^{(m+1)}$ are isomorphisms.
\end{proof}

The following lemma can be seen at once from the definition and the observation that $\mathcal{C}_{0, 0}^{(0)}$ and $\mathcal{C}_{0, 1}^{(0)}$ are naturally isomorphic to the monoid $\theta_3$ of integral homology $3$-spheres with the connected sum operation.
See Theorem \ref{thm_C} for the other cases.

\begin{lem} \label{lem_TC}
(i) $\mathcal{C}_{0, 0}^{(m)} = \mathcal{C}_{0, 1}^{(m)} = \mathcal{C}_{0, 0}^{(0)} = \mathcal{C}_{0, 1}^{(0)} = \theta_3$ for $m \geq 0.$

(ii) $\mathcal{C}_{0, 2}^{(m)} = \mathcal{C}_{0, 2}^{(1)}$ for $m \geq 1.$

(iii) $\mathcal{C}_{1, 0}^{(m)} = \mathcal{C}_{1, 0}^{(1)}$ for $m \geq 1.$
\end{lem}

The same argument as the proof of \cite[Proposition 2. 4]{GS2} gives the following proposition.

\begin{prop}
The monoid $\mathcal{C}_{g, n}^{(m)}$ is not finitely generated.
\end{prop}

In fact, we have an epimorphism $F \colon \mathcal{C}_{g, n}^{(m)} \to \theta_3$ as follows.
For $(M, i_{\pm}) \in \mathcal{C}_{g, n}^{(m)},$ we can write $M = M' \sharp M'',$ where $M'$ is the prime factor of $M$ containing $\partial M.$
Then $F(M, i_{\pm}) := M''.$
Therefore as pointed out in \cite{GS2} it is reasonable to consider the submonoid $\overline{\mathcal{C}}_{g, n}^{(m)}$ consisting of all $(M, i_{\pm}) \in\mathcal{C}_{g, n}^{(m)}$ with irreducible $M.$
Note that $\overline{\mathcal{C}}_{0, 0}^{(m)} = \emptyset$ and $\overline{\mathcal{C}}_{0, 1}^{(m)} = 1$ for all $m.$

Let $\zeta_i \in \pi$ be a representative of each boundary circle.
We denote by $\out^*(\Gamma_m)$ the group of outer automorphisms of $\Gamma_m$ which preserve the conjugacy class of $[\zeta_i] \in \Gamma_m$ for all $i.$
A homomorphism $\varphi_m \colon \mathcal{C}_{g,n}^{(m)} \to \out^*(\Gamma_m)$ is given by
\[ \varphi_m(M, i_{\pm}) := [(i_+)_*^{-1} \circ (i_-)_*]. \]
It is easily seen that $\varphi_m$ does not depend on the choices of a base point and $\zeta_i.$
We define
\begin{align*}
\mathcal{IC}_{g,n}^{(m)} &:= \ker \varphi_m, \\
\mathcal{I} \overline{\mathcal{C}}_{g,n}^{(m)} &:= \ker \varphi_m|_{\overline{\mathcal{C}}_{g,n}^{(m)}}.
\end{align*}

\begin{rem}
In \cite[Proposition 2.\ 3]{GS1} Goda and Sakasai showed that $\varphi_0(M, i_{\pm})$ preserves the intersection form of $\Sigma$ for all homology cylinders $(M, i_{\pm}).$
\end{rem}

\subsection{The homology cobordism groups of homology cylinders of higher-order}
Next we consider homology cobordisms for homology cylinders of higher-order.
See \cite{GL}, \cite{Le} for the case of ordinary homology cylinders.

\begin{defn} \label{defn_C}
For $(M, i_{\pm}), (N, j_{\pm}) \in \mathcal{C}_{g, n}^{(m)},$ we write $(M, i_{\pm}) \sim_m (N, j_{\pm})$ if there exists a compact oriented smooth $4$-manifold satisfying the following:

(i) $\partial W = M \cup_{i_+ \circ j_+^{-1}, i_- \circ j_-^{-1}} (-N),$

(ii) $H_*(M) \to H_*(W), H_*(N) \to H_*(W)$ are isomorphisms,

(iii) $\pi_1 M / (\pi_1 M)^{(m+1)} \to \pi_1 W / (\pi_1 W)^{(m+1)}, \pi_1 N / (\pi_1 N)^{(m+1)} \to \pi_1 W / (\pi_1 W)^{(m+1)}$ are isomorphisms.
\end{defn}

\begin{lem}
The relation $\sim_m$ is an equivalence relation on $\mathcal{C}_{g, n}^{(m)}$ which is compatible with the product operation.
\end{lem}

For the proof in the case where $m = 0,$ we refer to \cite[p.\ 246]{Le}.
Using an almost same technique as in the proof of Lemma \ref{lem_P}, we can also prove the lemma in the general case, and so we omit the proof.

We define $\mathcal{H}_{g, n}^{(m)} := \mathcal{C}_{g, n}^{(m)} / \sim_m,$ which has a natural group structure induced by the monoid structure of $\mathcal{C}_{g, n}^{(m)}.$
The inverse of $[M, i_{\pm}] \in \mathcal{H}_{g, n}^{(m)}$ is given by $[-M, i_{\mp}].$
There is a natural homomorphism $\mathcal{H}_{g, n}^{(m+1)} \to \mathcal{H}_{g, n}^{(m)}$ for each $m.$

\begin{rem}
The group $\mathcal{H}_{g, n}^{(0)}$ is nothing but the smooth homology cobordism group $\mathcal{H}_{g, n}$ of ordinary homology cylinders.
We can see that $\mathcal{H}_{0, 0}^{(0)}$ and $\mathcal{H}_{0, 1}^{(0)}$ are isomorphic to the smooth homology cobordism group $\Theta_3$ of integral homology $3$-spheres, and that $\mathcal{H}_{0, 2}^{(0)}$ is isomorphic to $\Z \oplus \mathcal{C}_{\Z},$ where $\mathcal{C}_{\Z}$ is the smooth concordance group of knots in integral homology $3$-spheres.
The $\Z$ factor comes from framings of knots.
See \cite[Section 2.\ 2]{CFK} and the references given there for more details.
\end{rem}

We can also consider topological homology cobordisms instead of smooth ones in Definition \ref{defn_C}.
Let $\mathcal{H}_{g, n}^{(m)~\text{top}}$ be the topological homology cobordism group of homology cylinders of order $m.$
Since we consider only topological methods, in fact, we can obtain analogous results on $\mathcal{H}_{g, n}^{(m)~\text{top}}$ to all the theorems on $\mathcal{H}_{g, n}^{(m)}$ in this paper.

Using the results of Freedman~\cite{Fre}, Furuta~\cite{Fu} and Fintushel and Stern~\cite{FS}, Cha-Friedl-Kim~\cite[Theorem 1.\ 1]{CFK} showed that the kernel of the epimorphism $\mathcal{H}_{g, n}^{(0)} \to \mathcal{H}_{g, n}^{(0)~\text{top}}$ contains an abelian group of infinite rank, which is given by the image of a homomorphism $\Theta_3 \to \mathcal{H}_{g, n}^{(0)}.$
Since the homomorphism $\Theta_3 \to \mathcal{H}_{g, n}^{(0)}$ factors through $\mathcal{H}_{g, n}^{(m)},$ the kernel of the epimorphism $\mathcal{H}_{g, n}^{(m)} \to \mathcal{H}_{g, n}^{(m)~\text{top}}$ also contains an abelian group of infinite rank.

The proof of the following lemma is straightforward from the definition and Lemma \ref{lem_TC}.
See Theorem \ref{thm_H} for general cases.

\begin{lem} \label{lem_TH}
(i) $\mathcal{H}_{0, 0}^{(m)} = \mathcal{H}_{0, 1}^{(m)} = \mathcal{H}_{0, 0}^{(0)} = \mathcal{H}_{0, 1}^{(0)} = \Theta_3$ for $m \geq 0.$

(ii) $\mathcal{H}_{0, 2}^{(m)} = \mathcal{H}_{0, 2}^{(1)}$ for $m \geq 1.$

(iii) $\mathcal{H}_{1, 0}^{(m)} = \mathcal{H}_{1, 0}^{(1)}$ for $m \geq 1.$
\end{lem}

\begin{prop}[\cite{CFK}, \cite{GL}, \cite{Le}]
The homomorphism $\mathcal{M}_{g, n} \to \mathcal{C}_{g, n}^{(0)} \to \mathcal{H}_{g, n}^{(0)}$ is injective.
\end{prop}

Since this injection factors through $\mathcal{H}_{g, n}^{(m)},$ we obtain the following corollary.

\begin{cor}
The homomorphism $\mathcal{M}_{g, n} \to \mathcal{C}_{g, n}^{(m)} \to \mathcal{H}_{g, n}^{(m)}$ is injective for all $m.$
\end{cor}

The homomorphism $\varphi_m \colon \mathcal{C}_{g,n}^{(m)} \to \out^*(\Gamma_m)$ induces a homomorphism $\mathcal{H}_{g,n}^{(m)} \to \out^*(\Gamma_m).$
By abuse of notation we also write $\varphi_m$ for the homomorphism.
We define
\[ \mathcal{IH}_{g,n}^{(m)} := \ker(\varphi_m \colon \mathcal{H}_{g,n}^{(m)} \to \out^*(\Gamma_m)). \]

\section{Reidemeister torsion homomorphisms} \label{sec_3}
\subsection{Reidemeister torsion}
First we review the definition of Reidemeister torsion over a skew field $\K$.
See \cite{Mi} and \cite{T1} for more details.

For a matrix over $\K$, we mean by an elementary row operation the addition of a left multiple of one row to another row.
After elementary row operations we can turn any matrix $A \in GL_k(\K)$ into a diagonal matrix $(d_{i, j})$.
Then the \textit{Dieudonn\'{e} determinant} $\det A$ is defined to be $[\prod_{i=1}^k d_{i, i}] \in \K_{ab}^{\times} := \K^{\times} / [\K^{\times}, \K^{\times}]$.

Let $C_* = (C_n \xrightarrow{\partial_n} C_{n-1} \to \cdots \to C_0)$ be a chain complex of finite dimensional right $\K$-vector spaces.
If we choose bases $b_i$ of $\im \partial_{i+1}$ and $h_i$ of $H_i(C_*)$ for $i = 0, 1, \dots n$, we can take a basis $b_i h_i b_{i-1}$ of $C_i$ as follows.
Picking a lift of $h_i$ in $\ker \partial_i$ and combining it with $b_i$, we first obtain a basis $b_i h_i$ of $C_i$.
Then picking a lift of $b_{i-1}$ in $C_i$ and combining it with $b_i h_i$, we can obtain a basis $b_i h_i b_{i-1}$ of $C_i$.

\begin{defn}
For given bases $\boldsymbol{c} = \{ c_i \}$ of $C_*$ and $\boldsymbol{h} = \{ h_i \}$ of $H_*(C_*)$, we choose a basis $\{ b_i \}$ of $\im \partial_*$ and define
\[ \tau(C_*, \boldsymbol{c}, \boldsymbol{h}) := \prod_{i=0}^n [b_i h_i b_{i-1} / c_i]^{(-1)^{i+1}} ~ \in \K_{ab}^{\times}, \]
where $[b_i h_i b_{i-1} / c_i]$ is the Dieudonn\'{e} determinant of the base change matrix from $c_i$ to $b_i h_i b_{i-1}$.
If $C_*$ is acyclic, then we write $\tau(C_*, \boldsymbol{c})$.
\end{defn}

It can be easily checked that $\tau(C_*, \boldsymbol{c}, \boldsymbol{h})$ does not depend on the choices of $b_i$ and $b_i h_i b_{i-1}$.

Torsion has the following multiplicative property.
Let
\[ 0 \to C_*' \to C_* \to C_*'' \to 0 \]
be a short exact sequence of finite chain complexes of finite dimensional right $\K$-vector spaces and let $\boldsymbol{c} = \{ c_i \}, \boldsymbol{c}' = \{ c_i' \}, \boldsymbol{c}'' = \{ c_i'' \}$ and $\boldsymbol{h} = \{ h_i \}, \boldsymbol{h}' = \{ h_i' \}, \boldsymbol{h}'' = \{ h_i'' \}$ be bases of $C_*, C_*', C_*''$ and $H_*(C_*), H_*(C_*'), H_*(C_*'')$.
Picking a lift of $c_i''$ in $C_i$ and combining it with the image of $c_i'$ in $C_i$, we obtain a basis $c_i' c_i''$ of $C_i$.
We denote by $\mathcal{H}_*$ the corresponding long exact sequence in homology, and by $\boldsymbol{d}$ the basis of $\mathcal{H}_*$ obtained by combining $\boldsymbol{h}, \boldsymbol{h}', \boldsymbol{h}''$.

\begin{lem}(\cite[Theorem 3.\ 1]{Mi}) \label{lem_M}
If $[c_i' c_i'' / c_i] = 1$ for all $i$, then
\[ \tau(C_*, \boldsymbol{c}, \boldsymbol{h}) = \tau(C_*', \boldsymbol{c}', \boldsymbol{h}') \tau(C_*'', \boldsymbol{c}'', \boldsymbol{h}'') \tau(\mathcal{H}_*, \boldsymbol{d}). \]
\end{lem} 

In the following when we write $C_*(\widetilde{X}, \widetilde{Y})$ for a connected finite CW-pair $(X, Y)$, $\widetilde{X}$, $\widetilde{Y}$ stand for the universal cover of $X$ and the pullback of $Y$ by the universal covering map $\widetilde{X} \to X$ respectively.
For a ring homomorphism $\varphi \colon \Z[\pi_1 X] \to \K$, we define the twisted homology group associated to $\varphi$ as follows:
\[ H_i^{\varphi}(X, Y; \K) := H_i(C_*(\widetilde{X}, \widetilde{Y}) \otimes_{\Z[\pi_1 X]} \K). \]

\begin{defn}
If $H_*^{\varphi}(X, Y; \K) = 0$, then we define the \textit{Reidemeister torsion} $\tau_{\varphi}(X, Y)$ associated to $\varphi$ as follows.
We choose a lift $\tilde{e}$ in $\widetilde{X}$ for each cell $e \subset X \setminus Y$. Then
\[ \tau_{\varphi}(X, Y) := [ \tau(C_*(\widetilde{X}, \widetilde{Y}) \otimes_{\Z[\pi_1 X]} \K, \langle \tilde{e} \otimes 1 \rangle_e)] \in \K_{ab}^{\times} / \pm \varphi(\pi_1 X). \]
\end{defn}

We can check that $\tau_{\varphi}(X, Y)$ does not depend on the choice of $\tilde{e}$.
It is known that Reidemeister torsion is a simple homotopy invariant of a finite CW-pair.

\subsection{Torsion of homology cylinders of higher-order}
Next we define non-commutative torsion invariants of homology cylinders of higher-order.
Torsion invariants of homology cylinders were first studied by Sakasai \cite{S1, S2}.

A group $G$ is called \textit{poly-torsion-free-abelian (PTFA)} if there exists a filtration
\[ 1 = G_0 \triangleleft G_1 \triangleleft \dots \triangleleft G_{n-1} \triangleleft G_n = G \]
such that $G_i / G_{i-1}$ is torsion free abelian.

\begin{prop}[\cite{Pa}]
If $G$ is a PTFA group, then $\Q[G]$ is a right (and left) Ore domain; namely $\Q[G]$ embeds in its classical right ring of quotients $\Q(G) := \Q[G](\Q[G] \setminus 0)^{-1}$.
\end{prop}

See \cite[Proposition 2.\ 10]{COT} for a proof of the following lemma.

\begin{lem} \label{lem_V}
For a CW-pair $(X, Y)$ and a homomorphism $\rho \colon \pi_1 X \to \Gamma$ to a PTFA group $\Gamma,$ if $H_*(X, Y; \Q) = 0,$ then $H_*^{\rho}(X, Y; \Q(\Gamma)) = 0.$ 
\end{lem}

For $(M, i_{\pm}) \in \mathcal{C}_{g, n}^{(m)},$ we denote by $\rho_m$ the pullback $\pi_1 M \to \Gamma_m$ of $(i_+)_*^{-1} \colon \pi_1 M / (\pi M)^{(m+1)} \to \Gamma_m.$
It is well-known that $\Gamma_m$ is torsion-free for all $m,$ and hence $\Gamma_m$ is PTFA for all $m.$ 
It follows from the above lemma and Remark \ref{rem_H} that $H_*^{\rho_m}(M, i_+(\Sigma); \Q(\Gamma_m)) = 0.$

\begin{defn}
We define a map $\tau_m \colon \mathcal{C}_{g, n}^{(m)} \to \Q(\Gamma_m)_{ab}^{\times} / \pm \Gamma_m$ by $\tau_m(M, i_{\pm}) := \tau_{\rho_m}(M, i_+(\Sigma)).$
\end{defn}

It can be easily computed that for all $\psi \in \mathcal{M}_{g, n}$ and all $m,$ $\tau_m(M(\psi)) = 1.$

\begin{lem} \label{lem_N}
For $(M, i_{\pm}) \in \mathcal{C}_{g, n}^{(m+1)}, H_*^{\rho_m}(M, i_+(\Sigma); \Z[\Gamma_m]) = 0.$
\end{lem}

\begin{proof}
Since $(i_+)* \colon \pi^{(m+1)} / \pi^{(m+2)} \to (\pi_1M)^{(m+1)} / (\pi_1M)^{(m+2)}$ is an isomorphism, $H_1^{\rho_k}(i_+(\Sigma); \Z[\Gamma_m]) \to H_1^{\rho_k}(M; \Z[\Gamma_m])$ is also an isomorphism.
Now the lemma follows from the long exact homology sequence for $(M, i_+(\Sigma))$
\end{proof}

We denote by $\wh(\Gamma) := K_1(\Z[\Gamma]) / \pm \Gamma$ the Whitehead group of a group $\Gamma.$
The Dieudonn\'{e} determinant induces a homomorphism $\wh(\Gamma) \to \Q(\Gamma)_{ab}^{\times} / \pm \Gamma.$

\begin{prop} \label{prop_V}
For all $(M, i_{\pm}) \in \mathcal{C}_{g, n}^{(m)}$ and all $k < m,$ $\tau_k(M, i_{\pm})$ is in the image of $\wh(\Gamma_k) \to \Q(\Gamma_k)_{ab}^{\times} / \pm \Gamma_k.$
Furthermore, for all $(M, i_{\pm}) \in \mathcal{C}_{g, n}^{(0)},$ $(M, i_{\pm}) \in \mathcal{C}_{g, n}^{(1)}$ if and only if $\tau_0(M, i_{\pm}) = 1.$
\end{prop}

\begin{proof}
First we suppose that $(M, i_{\pm}) \in \mathcal{C}_{g, n}^{(m)}.$
Since $M$ can be obtained by attaching the same number of $1$- and $2$-handles to $\Sigma \times [0, 1],$ $C_*^{\rho_k}(M, i_+(\Sigma); \Z[\Gamma_k])$ is simple homotopy equivalent to a chain complex
\[ 0 \to C_2 \xrightarrow{\partial} C_1 \to 0 \]
with $\rank C_1 = \rank C_2$ for $k \leq m.$
Hence $\tau_k(M, i_{\pm}) = [\det \partial]$ for $k \leq m.$
Since $H_1^{\rho_k}(M, i_+(\Sigma); \Z[\Gamma_k]) = 0$ for $k < m$ by Lemma \ref{lem_N}, $\partial$ is a surjection, and so an isomorphism for $k < m$, which proves the first statement.

Now the necessary condition in the second statement follows from the result by \cite{BHS} that for any free abelian group $\Gamma,$ $\wh(\Gamma) = 1.$

Next we suppose that $\tau_0(M, i_{\pm}) = 1$ for $(M, i_{\pm}) \in \mathcal{C}_{g, n}^{(0)}.$
Since $\det \partial \in \pm H_1(\Sigma),$ $\partial$ is an isomorphism for $k = 0.$
Hence $H_*^{\rho_0}(M, i_+(\Sigma); \Z[H_1(\Sigma)]) = 0.$
From the Poincar\'{e} duality and the universal coefficient theorem, we have $H_*^{\rho_0}(M, i_-(\Sigma); \Z[H_1(\Sigma)]) = 0.$
From the long exact homology sequence for $(M, i_{\pm}(\Sigma)),$ $H_1^{\rho_0}(i_{\pm}(\Sigma); \Z[H_1(\Sigma)]) \to H_1^{\rho_0}(M; \Z[H_1(\Sigma)])$ are isomorphisms, and so $(i_{\pm})_* \colon \pi^{(1)} / \pi^{(2)} \to (\pi_1 M)^{(1)} / (\pi_1 M)^{(2)}$ are also isomorphisms.
Therefore $(i_{\pm})_* \colon \pi / \pi^{(2)} \to \pi_1 M / (\pi_1 M)^{(2)}$ are isomorphisms, which gives the sufficient condition in the second statement.
\end{proof}

\begin{rem}
Though it is a well-known conjecture that for any finitely generated torsion-free group $\Gamma,$ $\wh(\Gamma) = 1,$ to author's knowledge there seems to be no appropriate reference on whether $\wh(\Gamma_m) = 1$ for $m > 0.$    
\end{rem}

Each $\varphi \in \out^*(\Gamma_m)$ induces an automorphism of $\Q(\Gamma_m)_{ab}^{\times} / \pm \Gamma_m,$ which we also denote by $\varphi.$
The following proposition is an extension of \cite[Proposition 3.\ 5]{CFK}. 
See also \cite[Proposition 6.\ 6]{S2} for a related result.

\begin{prop} \label{prop_P}
For $(M, i_{\pm}), (N, j_{\pm}) \in \mathcal{C}_{g, n}^{(m)},$
\[ \tau_m((M, i_{\pm}) \cdot (N, j_{\pm})) = \tau_m(M, i_{\pm}) \cdot \varphi_m(M, i_{\pm})(\tau_m(N, j_{\pm})). \]
\end{prop}

\begin{proof}
In all of the calculations below, we implicitly tensor the chain complexes with $\Q(\Gamma_m).$
We write $W := M \cup_{i_- \circ (j_+)^{-1}} N.$
We have the following short exact sequences:
\[ 0 \to C_*(\widetilde{j_+(\Sigma)}) \to C_*(\widetilde{M}, \widetilde{i_+(\Sigma)}) \oplus C_*(\widetilde{N}) \to C_*(\widetilde{W}, \widetilde{i_+(\Sigma)}) \to 0,  \]
\[ 0 \to C_*(\widetilde{j_+(\Sigma)}) \to C_*(\widetilde{N}) \to C_*(\widetilde{N}, \widetilde{j_+(\Sigma)}) \to 0,  \]
where we consider the homomorphisms
\begin{align*}
\rho &\colon \pi_1 W \to \pi_1 W / (\pi_1 W)^{(m+1)} \xrightarrow{(i_+)_*^{-1}} \Gamma_m, \\
\rho' &\colon \pi_1 N \to \pi_1 N / (\pi_1 N)^{(m+1)} \xrightarrow{\sim} \pi_1 W / (\pi_1 W)^{(m+1)} \xrightarrow{(i_+)_*^{-1}} \Gamma_m
\end{align*}
respectively.
It follows from the long exact homology sequence that the inclusion map $j_+(\Sigma) \to N$ induces an isomorphism $H_*^{\rho'}(j_+(\Sigma); \Q(\Gamma_m)) \to H_*^{\rho'}(N; \Q(\Gamma_m))$.
Let $\boldsymbol{c}, \boldsymbol{c}'$ be bases of $C_*(\widetilde{N}), C_*(\widetilde{j_+(\Sigma)})$ consisting of cells and let $\boldsymbol{h}, \boldsymbol{h}'$ be bases of $H_*^{\rho'}(N; \Q(\Gamma_m)), H_*^{\rho'}(j_+(\Sigma); \Q(\Gamma_m))$ such that $\boldsymbol{h}$ is the image of $\boldsymbol{h}'$ by the isomorphism.
By Lemma \ref{lem_M} we obtain the following equations:
\begin{align*}
\tau_{\rho}(M, i_+(\Sigma)) \cdot [\tau(C_*(\widetilde{N}), \boldsymbol{c}, \boldsymbol{h})] &= [\tau(C_*(\widetilde{j_+(\Sigma)}), \boldsymbol{c}', \boldsymbol{h}')] \cdot \tau_{\rho}(W, i_+(\Sigma)), \\ 
[\tau(C_*(\widetilde{N}), \boldsymbol{c}, \boldsymbol{h})] &= [\tau(C_*(\widetilde{j_+(\Sigma)}), \boldsymbol{c}', \boldsymbol{h}')] \cdot \tau_{\rho'}(N, j_+(\Sigma)). 
\end{align*}
Hence
\[ \tau_{\rho}(W, i_+(\Sigma)) = \tau_{\rho}(M, i_+(\Sigma)) \cdot \tau_{\rho'}(N, j_+(\Sigma)). \]
By the functoriality of Reidemeister torsion $\tau_{\rho'}(N, j_+(\Sigma)) = \varphi_m(M, i_{\pm})(\tau_m(N, j_{\pm})),$ which establishes the formula.
\end{proof}

\begin{cor} \label{cor_H1}
The map $\tau_m \rtimes \varphi_m \colon \mathcal{C}_{g, n}^{(m)} \to (\Q(\Gamma_m)_{ab}^{\times} / \pm \Gamma_m) \rtimes \out^*(\Gamma_m)$ is a homomorphism.
\end{cor}

\begin{cor} \label{cor_H2}
The map $\tau_m \colon \mathcal{IC}_{g, n}^{(m)} \to \Q(\Gamma_m)_{ab}^{\times} / \pm \Gamma_m$ is a homomorphism.
\end{cor}

\subsection{Torsion and homology cobordisms}
We define the involution $\gamma \mapsto \overline{\gamma}$ on $\Gamma_m$ by $\overline{\gamma} = \gamma^{-1}$ and naturally extend it to $\Q(\Gamma_m)$ for each $m.$

The following theorem is an extension of \cite[Theorem 3.\ 10]{CFK}. 
 
\begin{thm}
Let $(M, i_{\pm}), (N, j_{\pm}) \in \mathcal{C}_{g, n}^{(m)}.$
If $(M, i_{\pm}) \sim_m (N, j_{\pm}),$ then
\[ \tau_m(M, i_{\pm}) = \tau_m(N, j_{\pm}) \cdot q \cdot \bar{q} \]
for some $q \in \Q(\Gamma_m)_{ab}^{\times} / \pm \Gamma_m.$
\end{thm}

\begin{proof}
We pick a homology cobordism $W$ between $M$ and $N$ as in Definition \ref{defn_C}.
Let $\rho$ be the homomorphism $\pi_1 W \to \pi_1 W / (\pi_1 W)^{(m+1)} \xrightarrow{(i_+)_*^{-1}} \Gamma_m.$
The long exact homology sequences for $(W, M), (W, N), (W, i_+(\Sigma))$ give $H_*(W, M) = H_*(W, N) = H_*(W, i_+(\Sigma)) = 0.$
By Lemma \ref{lem_V} we obtain $H_*^{\rho}(W, M; \Q(\Gamma_m)) = H_*^{\rho}(W, N; \Q(\Gamma_m)) = H_*^{\rho}(W, i_+(\Sigma); \Q(\Gamma_m)) = 0.$
By applying Lemma \ref{lem_M} to the following exact sequence
\[ 0 \to C_*(\widetilde{M}, \widetilde{i_+(\Sigma)}) \otimes \Q(\Gamma_m) \to C_*(\widetilde{W}, \widetilde{i_+(\Sigma)}) \otimes \Q(\Gamma_m) \to C_*(\widetilde{W}, \widetilde{M}) \otimes \Q(\Gamma_m) \to 0, \]
we get
\[ \tau_{\rho}(W, i_+(\Sigma)) = \tau_{\rho}(M, i_+(\Sigma)) \cdot \tau_{\rho}(W, M). \]
Similarly,
\[ \tau_{\rho}(W, i_+(\Sigma)) = \tau_{\rho}(N, j_+(\Sigma)) \cdot \tau_{\rho}(W, N). \]
By the duality of Reidemeister torsion
\[ \tau_{\rho}(W, M) = \overline{\tau_{\rho}(W, N)}^{-1} \]
(e.g., see \cite{CF, KL, Mi}).
Hence
\[ \tau_{\rho}(M, i_+(\Sigma)) = \tau_{\rho}(N, j_+(\Sigma)) \cdot \tau_{\rho}(W, N) \cdot  \overline{\tau_{\rho}(W, N)}, \]
which proves the theorem.
\end{proof}

We set
\[ N_m := \{ \pm \gamma \cdot q \cdot \overline{q} \in \Q(\Gamma_m)_{ab}^{\times} ~;~ \gamma \in \Gamma_m, q \in \Q(\Gamma_m)_{ab}^{\times} \}. \]

\begin{cor} \label{cor_H3}
The map $\tau_m \rtimes \varphi_m \colon \mathcal{H}_{g, n}^{(m)} \to (\Q(\Gamma_m)_{ab}^{\times} / N_m) \rtimes \out^*(\Gamma_m)$ is a homomorphism.
\end{cor}

\begin{cor} \label{cor_H4}
The map $\tau_m \colon \mathcal{IH}_{g, n}^{(m)} \to \Q(\Gamma_m)_{ab}^{\times} / N_m$ is a homomorphism.
\end{cor}

The following theorem showed that if $(g, n) \neq (0, 0), (0, 1)$ and $m > 0,$ then $\mathcal{H}_{g, n}^{(m)}$ is another enlargement of $\mathcal{M}_{g, n}.$

\begin{thm} \label{thm_H}
If $(g, n) \neq (0, 0), (0, 1),$ then the homomorphisms $\mathcal{H}_{g, n}^{(m)} \to \mathcal{H}_{g, n}^{(0)}, \mathcal{IH}_{g, n}^{(m)} \to \mathcal{IH}_{g, n}^{(0)}$ are not surjective for $m > 0.$
\end{thm}

\begin{proof}
By Proposition \ref{prop_V} the image of the composition of $\mathcal{H}_{g, n}^{(m)} \to \mathcal{H}_{g, n}^{(0)}$ and $\tau_0 \rtimes \varphi_0 \colon \mathcal{H}_{g, n}^{(0)} \to (\Q(\Gamma_0)_{ab}^{\times} / N_0) \rtimes \out^*(\Gamma_0)$ is contained in $1 \times \out^*(\Gamma_0)$ and that of  $\mathcal{IH}_{g, n}^{(m)} \to \mathcal{IH}_{g, n}^{(0)}$ and $\tau_0 \colon \mathcal{IH}_{g, n}^{(0)} \to \Q(\Gamma_0)_{ab}^{\times} / N_0$ is trivial.
On the other hand, in \cite{CFK} Cha, Friedl and Kim detected elements of the image of $\tau_0 \rtimes \varphi_0 \colon \mathcal{H}_{g, n}^{(0)} \to (\Q(\Gamma_0)_{ab}^{\times} / N_0) \rtimes \out^*(\Gamma_0)$ not contained in $1 \times \out^*(\Gamma_0)$ and nontrivial ones of $\tau_0 \colon \mathcal{IH}_{g, n}^{(0)} \to \Q(\Gamma_0)_{ab}^{\times} / N_0$ when $(g, n) \neq (0, 0), (0, 1).$
These prove the theorem.
\end{proof}

\begin{rem} \label{rem_R}
\begin{enumerate}
\item In fact one can say more about the cokernels of the homomorphisms:
It follows from the above argument in the proof and \cite[Theorems 1.\ 2, 1.\ 3, 7.\ 2]{CFK} the cokernel of the homomorphism $\mathcal{H}_{g, n}^{(m)} \to \mathcal{H}_{g, n}^{(0)}$ contains a direct summand isomorphic to $(Z/2)^\infty$ if $(g, n) \neq (0, 0), (0, 1)$ and one isomorphic to $\Z^\infty$ if $n > 1$, and the cokernel of the homomorphism $\mathcal{IH}_{g, n}^{(m)} \to \mathcal{IH}_{g, n}^{(0)}$ contains a direct summand isomorphic to $(Z/2)^\infty$ if $(g, n) \neq (0, 0), (0, 1)$ and one isomorphic to $\Z^\infty$ if $g > 1$ or $n > 1$.
\item It is an important question whether the homomorphisms $\mathcal{H}_{g, n}^{(m)} \to \mathcal{H}_{g, n}^{(0)}, \mathcal{IH}_{g, n}^{(m)} \to \mathcal{IH}_{g, n}^{(0)}$ are in general injective or not.
 \end{enumerate}
\end{rem}

\section{Construction and computation} \label{sec_4}
For nontrivial $\gamma \in \pi$ and a tame knot $K \subset S^3,$ we construct a homology cylinder $M(\gamma, K)$ as follows.
See \cite[Section 4]{CFK} for various constructions of homology cylinders.

Let $* \in \Sigma$ be the base point for $\pi.$
We choose a smooth path $f \colon [0, 1] \to \Sigma$ representing $\gamma$ such that $f^{-1}(*) = \{ 0, 1 \},$ and define $\tilde{f} \colon [0, 1] \to \Sigma \times [0, 1],$ $c \colon [0, 1] \to \Sigma \times [0, 1]$ by $\tilde{f}(t) = (f(t), t)$ and $c(t) = (*, 1-t).$
After pushed into the interior, $\tilde{f} \cdot c$ determines a tame knot $J \subset \Int M(id)$.
Let $E_J$ be the complement of an open tubular neighborhood $Z$ of $J$.
We take a framing of $J$ so that a meridian of $J$ represents the conjugacy class of the generator of the kernel of $\pi_1 \partial Z \to H_1(M(id))$ compatible with the orientation of $J$ and that a longitude of $J$ represents the conjugacy class of the image of $\gamma$ by $(i_-)_* \colon \pi \to \pi_1 E_J$.
Let $E_K$ be the exterior of $K$.
Now $M(\gamma, K)$ is the result of attaching $-E_K$ to $E_J$ along the boundaries so that a longitude and a meridian of $K$ correspond to a meridian and a longitude of $J$ respectively.
Note that if $(g, n) = (0, 0)$ or $(0, 1),$ then $M(1, K) = M(id)$ for all $K$.

\begin{prop} \label{prop_C1}
If $(g, n) \neq (0, 0), (0, 1)$ and $\gamma \in \pi^{(m)} \setminus 1,$ then $M(\gamma, K) \in \mathcal{I} \overline{\mathcal{C}}_{g, n}^{(m)}$ for all $K.$
\end{prop}

\begin{proof}
If $K$ is a trivial knot, then $M(\gamma, K) = M(id) \in \mathcal{I} \overline{\mathcal{C}}_{g, n}^{(m)}$ for all nontrivial $\gamma \in \pi$ and all $m$.
In the following we assume that $K$ is nontrivial.

Since $E_J$ and $E_K$ are both irreducible and $\partial Z$ and $\partial E_K$ are both incompressible, $M(\gamma, K)$ is also irreducible.

Extending a degree $1$ map $(E_K, \partial E_K ) \to (Z, \partial Z)$ by the identity map on $E_J$, we have $f \colon M(\gamma, K) \to M(id).$
We show that $f_* \colon \pi_1 M(\gamma, K) / (\pi_1 M(\gamma, K))^{(m+1)} \to \pi_1 M(id) / (\pi_1 M(id))^{(m+1)}$ is an isomorphism, which immediately gives the desired conclusion from the following commutative diagram:
\[ \xymatrix{
& \pi_1 M(\gamma, K) / (\pi_1 M(\gamma, K))^{(m+1)} \ar[dd]^{f_*} & \\
\Gamma_m \ar[ru]^{(i_+)_*} \ar[rd]_{(i_-)_*} & & \Gamma_m \ar[lu]_{(i_+)_*} \ar[ld]^{(i_-)_*} \\
& \pi_1 M(id) / (\pi_1 M(id))^{(m+1)} &
} \]

Let $\lambda_J, \mu_J \in \pi_1 E_J$ and $\lambda_K, \mu_K \in \pi_1 E_K$ be longitude-meridian pairs. 
By the van Kampen theorem $\pi_1 M(\gamma, K)$ is the amalgamated product of $\pi_1 E_J$ and $\pi_1 E_K$ with $\lambda_J = \mu_K$ and $\mu_J = \lambda_K$, and $\pi_1 M(id)$ is that of  $\pi_1 E_J$ and $\langle t \rangle$ with $\lambda_J = t$ and $\mu_J = 1.$
Here $f_* \colon \pi_1 M(\gamma, K) \to \pi_1 M(id)$ is the identity map on $\pi_1 E_J$ and is the Hurewicz map on $\pi_1 E_K.$
Hence the kernel is the normal closure of $(\pi_1 E_K)^{(1)}$ in $\pi_1 M(\gamma, K).$
Thus it suffices to show that $\pi_1 E_K \subset (\pi_1 M(\gamma, K))^{(m)}.$
Since $\pi_1 E_K$ is normally generated by $\mu_K,$ it suffices to show that $\mu_K \in (\pi_1 M(\gamma, K))^{(m)}.$

Suppose that $\mu_K \in (\pi_1 M(\gamma, K))^{(k)}$ for an integer $k < m.$
Since $\gamma \in \pi^{(m)},$ a longitude of $J$ bounds a map from a symmetric $m$-stage grope in $M(id)$ such that the grope stages meet $J$ transversely (e.g., see \cite{CT}).
Hence it bounds a map from a punctured symmetric $m$-stage grope in $E_J,$ where the boundaries of these punctures are meridians of $E_J.$
Therefore for some $\xi_i \in \pi_1 E_J,$
\[ \lambda_J = \prod_i \xi_i \mu_J^{\pm 1} \xi_i^{-1} \prod_j [a_j, b_j], \]
where representatives of $a_j, b_j \in \pi_1 E_J$ bound a map from punctured $(m-1)$-stage gropes in $E_J.$
Hence $a_j, b_j$ have similar expressions as $\lambda_J.$
Continuing in this fashion, we see from $\mu_J = \lambda_K \in (\pi_1 M(\gamma, K))^{(k+1)}$ that $\mu_K = \lambda_J \in (\pi_1 M(\gamma, K))^{(k+1)}.$
It follows by induction that $\mu_K \in (\pi_1 M(\gamma, K))^{(m)}.$
\end{proof}

\begin{prop} \label{prop_C2}
Let $\gamma \in \pi^{(m)}.$
Then $\tau_m(M(\gamma, K)) =  [\Delta_K(\gamma)]$ for all $K.$
\end{prop}

\begin{proof}
In all of the calculations below, we implicitly tensor the chain complexes with $\Q(\Gamma_m).$

First we suppose that $[\gamma] = 1 \in \Gamma_m.$
We have the following short exact sequences:
\[ 0 \to C_*(\widetilde{\partial E_K}) \to C_*(\widetilde{E_J}, \widetilde{i_+(\Sigma)}) \oplus C_*(\widetilde{E_K}) \to C_*(\widetilde{M(\gamma, K)}, \widetilde{i_+(\Sigma)}) \to 0, \]
\[ 0 \to C_*(\widetilde{\partial Z}) \to C_*(\widetilde{E_J}, \widetilde{i_+(\Sigma)}) \oplus C_*(\widetilde{Z}) \to C_*(\widetilde{M(id)}, \widetilde{i_+(\Sigma)}) \to 0, \]
where we consider  $\rho_m \colon \pi_1 M(\gamma, K) \to \Gamma_m,$ $\rho_m' \colon \pi_1 M(id) \to \Gamma_m.$
Let $f \colon M(\gamma, K) \to M(id)$ be the map taken in the proof of Proposition \ref{prop_C1}.
It is easily seen that the induced maps $H_*^{\rho_m}(\partial E_K; \Q(\Gamma_m)) \to H_*^{\rho_m'}(\partial Z; \Q(\Gamma_m)),$ $H_*^{\rho_m}(E_K; \Q(\Gamma_m)) \to H_*^{\rho_m'}(Z; \Q(\Gamma_m))$ are isomorphisms.
We pick bases $\boldsymbol{h}, \boldsymbol{h}', \boldsymbol{h}''$ of $H_*^{\rho_m}(\partial E_K; \Q(\Gamma_m)), H_*^{\rho_m}(E_J, i_+(\Sigma); \Q(\Gamma_m)), H_*^{\rho_m}(E_K; \Q(\Gamma_m))$ respectively such that the isomorphism $H_*^{\rho_m}(\partial E_K; \Q(\Gamma_m)) \to H_*^{\rho_m}(E_J, i_+(\Sigma); \Q(\Gamma_m)) \oplus H_*^{\rho_m}(E_K; \Q(\Gamma_m))$ maps $\boldsymbol{h}$ to $\boldsymbol{h}' \oplus \boldsymbol{h}''.$
By Lemma \ref{lem_M} we obtain
\begin{align*}
[\tau(C_*(\widetilde{E_J}, \widetilde{i_+(\Sigma)}), \boldsymbol{h}')] \cdot [\tau(C_*(\widetilde{E_K}), \boldsymbol{h}'')] &= [\tau(C_*(\widetilde{\partial E_K}), \boldsymbol{h})] \cdot \tau_{\rho_m}(M(\gamma, K), i_+(\Sigma)), \\
[\tau(C_*(\widetilde{E_J}, \widetilde{i_+(\Sigma)}), \boldsymbol{h}')] \cdot [\tau(C_*(\widetilde{Z}), f_*(\boldsymbol{h}''))] &= [\tau(C_*(\widetilde{\partial Z}), f_*(\boldsymbol{h}))] \cdot \tau_{\rho_m'}(M(id), i_+(\Sigma)),
\end{align*}
where we consider bases of chain complexes consisting of cells and the notation of these bases is omitted. 
Since
\begin{align*}
[\tau(C_*(\widetilde{\partial E_K}), \boldsymbol{h})] &= [\tau(C_*(\widetilde{\partial Z}), f_*(\boldsymbol{h}))], \\
[\tau(C_*(\widetilde{E_K}), \boldsymbol{h}'')] &= [\tau(C_*(\widetilde{Z}), f_*(\boldsymbol{h}''))], \\
\tau_{\rho_m'}(M(id), i_+(\Sigma)) &= 1,
\end{align*}
we have
\[ \tau_{\rho_m}(M(\gamma, K), i_+(\Sigma)) = 1 = [\Delta_K(\gamma)]. \]

Next we suppose that $[\gamma] \neq 1 \in \Gamma_m.$
In this case $H_*^{\rho_m}(\partial E_K; \Q(\Gamma_m)), H_*^{\rho_m}(E_K; \Q(\Gamma_m)), H_*^{\rho_m'}(\partial Z; \Q(\Gamma_m)), H_*^{\rho_m'}(Z; \Q(\Gamma_m))$ vanish.
Therefore $H_*^{\rho_m}(E_J, i_+(\Sigma); \Q(\Gamma_m))$ also vanishes.
By Lemma \ref{lem_M} we obtain
\begin{align*}
\tau_{\rho_m}(E_J, i_+(\Sigma)) \cdot \tau_{\rho_m'}(E_K) &= \tau_{\rho_m}(\partial E_K) \cdot \tau_{\rho_m}(M(\gamma, K), i_+(\Sigma)), \\
\tau_{\rho_m'}(E_J, i_+(\Sigma)) \cdot \tau_{\rho_m'}(Z) &= \tau_{\rho_m'}(\partial Z) \cdot \tau_{\rho_m'}(M(id), i_+(\Sigma)).
\end{align*}
Here
\begin{align*}
&\tau_{\rho_m}(E_K) = [\Delta_K(\gamma) (\gamma - 1)^{-1}], \\
&\tau_{\rho_m'}(Z) = [(\gamma - 1)^{-1}], \\
&\tau_{\rho_m}(\partial E_K) = \tau_{\rho_m'}(\partial Z) = \tau_{\rho_m'}(M(id), i_+(\Sigma)) = 1,
\end{align*}
which are easy to check.
Now these equations give the desired formula.
\end{proof}

\begin{rem}
In the proof we also care about the case where $[\gamma] = 1$, and also in this point Proposition \ref{prop_C2} is a generalization of \cite[Proposition 4.\ 3]{CFK}.
\end{rem}

Considering the homology long exact sequences of the chain complexes with $\Z[\Gamma]$ coefficients instead of Lemma \ref{lem_M} in the proof, we obtain the following lemma.
\begin{lem} \label{lem_A}
Let $\gamma \in \pi^{(m)} \setminus 1.$
Then
\[ H_1^{\rho_m}(M(\gamma, K), i_+(\Sigma); \Z[\Gamma]) \cong \mathcal{A}_K \otimes_{\Z[t, t^{-1}]} \Z[\Gamma_m], \]
where $\mathcal{A_K}$ is the Alexander module of $K$ and we consider the homomorphism $\Z[t, t^{-1}] \to \Z[\Gamma_m]$ defined by $t \mapsto \gamma.$ 
\end{lem}

Now we are in position to show the difference between $\mathcal{C}_{g, n}^{(m)}$ and $\mathcal{C}_{g, n}^{(m+1)}.$
\begin{thm} \label{thm_C}
(i) $\mathcal{I} \overline{\mathcal{C}}_{0, 2}^{(1)} \neq \mathcal{I} \overline{\mathcal{C}}_{0, 2}^{(0)}.$

(ii) $\mathcal{I} \overline{\mathcal{C}}_{1, 0}^{(1)} \neq \mathcal{I} \overline{\mathcal{C}}_{1, 0}^{(0)}.$

(iii) If $(g, n) \neq (0, 0), (0, 1), (0, 2), (1, 0),$ then $\mathcal{I} \overline{\mathcal{C}}_{g, n}^{(m+1)} \neq \mathcal{I} \overline{\mathcal{C}}_{g, n}^{(m)}$ for all $m.$
\end{thm}

\begin{proof}
Suppose that $\pi^{(m+1)} \neq \pi^{(m)}.$
Let $\gamma \in \pi^{(m)} \setminus \pi^{(m+1)}$ and let $K \subset S^3$ be a tame knot with nontrivial $\mathcal{A}_K.$
By Proposition \ref{prop_C1} we see $M(\gamma, K) \in \mathcal{I} \overline{\mathcal{C}}_{g, n}^{(m)}.$
By Lemma \ref{lem_A} we have $H_1^{\rho_m}(M(\gamma, K), i_+(\Sigma); \Z[\Gamma_m]) \neq 0.$
On the other hand, $H_1^{\rho_m}(M, i_+(\Sigma); \Z[\Gamma_m]) = 0$ for every $(M, i_{\pm}) \in \mathcal{C}_{g, n}^{(m+1)}$
(e.g., see the proof of Proposition \ref{prop_V}).
Therefore 
 $M(\gamma, K) \notin \mathcal{I} \overline{\mathcal{C}}_{g, n}^{(m+1)},$
which gives the theorem.
\end{proof}

\section{Reduction of the torsion group} \label{sec_5}
A \textit{bi-order} $\leq$ of a group $\Gamma$ is a total order of $\Gamma$ satisfying that if $x \leq y,$ then $a x b \leq a y b$ for all $a, b, x, y \in \Gamma.$
A group $\Gamma$ is called \textit{bi-orderable} if $\Gamma$ admits a bi-order.
It is well-known that an abelian group is bi-orderable if and only if it is torsion-free.
The following lemma is an immediate consequence of  \cite[Corollary 2.\ 4.\ 2, Corollary 2.\ 4.\ 3]{MR}.

\begin{lem}
For a free group $F$, $F / F^{(m)}$ is bi-orderable for all $m.$
\end{lem}

\begin{rem}
It is well-known that every finitely generated torsion-free nilpotent group is residually $p$ for any prime $p.$
Rhemtulla~\cite{R} showed that a group which is residually $p$ for infinitely many $p$ is bi-orderable.
To author's knowledge there seems to be no appropriate reference on whether $\Gamma_m$ is residually nilpotent, residually p for infinitely many $p$ or bi-orderable in the case where $m > 0$ and $n = 0.$ 
\end{rem}

In the following we assume $n > 0$ and fix a bi-order of $\Gamma_{m-1}.$
Let $A_m := \pi^{(m)} / \pi^{(m+1)}$ and let $C_m$ be the subgroup of $\Q(A_m)^{\times}$ generated by
\[ \left\{ \pm a \cdot \frac{\gamma p \gamma^{-1}}{p} \in \Q(A_m)^{\times} ~;~ a \in A_m, \gamma \in \Gamma_m, p \in \Q(A_m)^{\times} \right\}. \]
We define a map $d \colon \Z[\Gamma_m] \setminus 0 \to \Q(A_m)^{\times} / C_m$ by
\[ d \left( \sum_{\delta \in \Gamma_{m-1}} \sum_{\gamma \in \Gamma_m, [\gamma] = \delta} a_{\gamma} \gamma \right) = \left[ \left( \sum_{\gamma \in \Gamma_m, [\gamma] = \delta_{max}} a_{\gamma} \gamma \right) \gamma_0^{-1} \right],  \]
where $\delta_{\max} \in \Gamma_{m-1}$ is the maximum with respect to the fixed bi-order such that for some $\gamma \in \Gamma_m$ with $[\gamma] = \delta_{max},$ $a_{\gamma} \neq 0,$ and $\gamma_0 \in \Gamma_m$ is an element with $[\gamma_0] = \delta_{max}.$
The proof of the following lemma is straightforward.

\begin{lem}
The map $d \colon \Z[\Gamma_m] \setminus 0 \to \Q(A_m)^{\times} / C_m$ does not depend on the choice of $\gamma_0$ and is a monoid homomorphism.
\end{lem}

By the lemma we have a group homomorphism $\Q(\Gamma_m)^{\times}_{ab} / \pm \Gamma_m \to \Q(A_m)^{\times} / C_m$ which maps $f \cdot g^{-1}$ to $d(f) \cdot d(g)^{-1}$ for $f, g \in \Z[\Gamma_m] \setminus 0.$
By abuse of notation, we use the same letter $d$ for the homomorphism.
Since there is a natural section $\Q(A_m)^{\times} / C_m \to \Q(\Gamma_m)^{\times}_{ab} / \pm \Gamma_m$ of $d,$ $\Q(A_m)^{\times} / C_m$ can be seen as a direct summand of $\Q(\Gamma_m)^{\times}_{ab} / \pm \Gamma_m$.

For irreducible $p, q \in \Z[A_m] \setminus 0,$ we write $p \sim q$ if there exist $a \in A_m$ and $\gamma \in \Gamma_m$ such that $p = \pm a \cdot \gamma q \gamma^{-1}.$
Since $\Z[A_m]$ is a unique factorization domain, every $x \in \Q(A_m)^{\times} / C_m$ can be written as $x = \prod_{[p]} [p]^{e_{[p]}},$ where $e_{[p]}$ is a uniquely determined integer.
Let $e \colon \Q(A_m)^{\times} / C_m \to \oplus_{[p]} \Z$ be the isomorphism given by $x \mapsto \sum_{[p]} e_{[p]}.$

Recall that for every monoid $S$, there exists a monoid homomorphism $g \colon S \to \mathcal{U}(S)$ to a group $\mathcal{U}(S)$ satisfying the following:
For every monoid homomorphism $f \colon S \to G$ to a group $G$, there exists a unique group homomorphism $f' \colon \mathcal{U}(S) \to G$ such that $f = f' \circ g$.
By the universality $\mathcal{U}(S)$ is uniquely determined up to isomorphisms. 
The following theorem is an analogous result of Goda and Sakasai in \cite{GS2}.

\begin{thm} \label{thm_A}
If $n > 0$ and $(g, n) \neq (0, 1), (0, 2),$ then the abelianization of $\mathcal{U}(\mathcal{I} \overline{\mathcal{C}}_{g, n}^{(m)})$ has infinite rank for all $m.$
\end{thm}

\begin{proof}
Let $\gamma \in \pi^{(m)} \setminus \pi^{(m+1)}$ and let $K \subset S^3$ be a tame knot.
By Proposition \ref{prop_C1} we see $M(\gamma, K) \in \mathcal{I} \overline{\mathcal{C}}_{g, n}^{(m)}.$
By Propositions \ref{prop_P}, \ref{prop_C2} we have
\[ d \circ \tau_m(M(\gamma, K)) = [\Delta_K(\gamma)]. \]
Since it is well-known that for any $p \in \Z[t, t^{-1}]$ with $p(t^{-1}) = p(t)$ and $p(1) = 1,$ there exists a knot $K \subset S^3$ such that $\Delta_K = p,$ the image of $e \circ d \circ \tau_m \colon \mathcal{I} \overline{\mathcal{C}}_{g, n}^{(m)} \to \oplus_{[p]} \Z$ contains a submonoid isomorphic to $\Z_{\geq 0}^{\infty}.$
Therefore the image of the induced map $\mathcal{U}(\mathcal{I} \overline{\mathcal{C}}_{g, n}^{(m)}) \to \oplus_{[p]} \Z$ is a free abelian group of infinite rank, which proves the theorem.
\end{proof}

\begin{rem}
If we suppose that $\wh(\Gamma_m) = 1,$ we could conclude by Proposition \ref{prop_V} that under the same assumption $\mathcal{U}(\mathcal{I} \overline{\mathcal{C}}_{g, n}^{(m)}) / \mathcal{U}(\mathcal{I} \overline{\mathcal{C}}_{g, n}^{(m+1)})$ should have infinite rank for all $m.$
\end{rem}

\begin{cor} \label{cor_A}
If $n > 0$ and $(g, n) \neq (0, 1), (0, 2),$ then $\mathcal{I} \overline{\mathcal{C}}_{g, n}^{(m)}$ is not finitely generated for all $m.$
\end{cor}

We conclude with an observation concerning the abelianization of $\mathcal{IH}_{g, n}^{(m)}.$
We set
\[ N_m' := \{ \pm a \cdot q \cdot \overline{q} \in \Q(A_m)^{\times} ~;~ a \in A_m, q \in \Q(A_m)^{\times} \}. \]
There is a natural map $\iota \colon \Q(A_m)^{\times} / (C_m \cdot N_m') \to \Q(\Gamma_m)_{ab}^{\times} / N_m.$
From the unique factorization property of $\Z[A_m]$ we have the isomorphism $e' \colon \Q(A_m)^{\times} / (C_m \cdot N_m') \to (\oplus_{[p] = [\bar{p}]} \Z / 2) \oplus (\oplus_{[p] \neq [\bar{p}]} \Z)$ induced by $e \colon \Q(A_m)^{\times} / C_m \to \oplus_{[p]} \Z.$
If $n > 0$ and $(g, n) \neq (0, 1), (0, 2),$ then from the argument in the proof of Theorem \ref{thm_A} the image of $\tau_m \colon \mathcal{IH}_{g, n}^{(m)} \to \Q(\Gamma_m)_{ab}^{\times} / N_m$ contains the image of a direct summand isomorphic to $(\Z / 2)^{\infty}$ by $\iota.$
Thus to investigate $\ker \iota$ is essential to detect size of the image of $\tau_m.$ \\

\noindent
\textbf{Acknowledgment.}
The author wishes to express his gratitude to Toshitake Kohno and Tomotada Ohtsuki for their encouragement and helpful suggestions.
A part of this project was advanced while the author visited the Mathematical Institute of the University of Cologne and he is especially grateful to Stefan Friedl.
The author would also like to thank Hiroshi Goda, Tetsuya Ito, Wolfgang L\"{u}ck, Takayuki Morifuji, Andrei Pajitnov, Mark Powell, Jean Raimbault and Takuya Sakasai for fruitful discussions.
Finally, the author would like to thank the anonymous referee for helpful comments, in particular the one about Remark \ref{rem_R}.
This research was supported by JSPS Research Fellowships for Young Scientists.


\end{document}